\documentclass{amsart}
\usepackage{amssymb}
\usepackage{amscd}
\usepackage{graphicx}
\usepackage{amsmath}
\usepackage{setspace}
\usepackage[all]{xy}
\usepackage{color}
\usepackage{enumerate}

\newtheorem{theorem}{Theorem}[section]

\newtheorem{corollary}[theorem]{Corollary}

\newtheorem{lemma}[theorem]{Lemma}

\newtheorem{proposition}[theorem]{Proposition}

\theoremstyle{remark}
\newtheorem{definition}[theorem]{Definition}

\newtheorem{remark}[theorem]{Remark}

\newcommand{\CB}{\mathcal{B}}

\newcommand{\CM}{\mathcal{M}}

\newcommand{\CX}{\mathcal{X}}
\newcommand{\CY}{\mathcal{Y}}

\newcommand{\Hom}{\mathrm{Hom}}

\newcommand{\Ker}{\mathrm{Ker}}
\newcommand{\Coker}{\mathrm{Coker}}
\newcommand{\Ext}{\mathrm{Ext}}

\newcommand{\End}{\mathrm{End}}

\newcommand{\Cogen}{\mathrm{Cogen}}

\newcommand{\presr}{\mathrm{pres}\text{-}}
\newcommand{\presl}{\text{-}\mathrm{pres}}

\newcommand{\Refl}{\mathrm{ Refl}}

\newcommand{\Add}{\mathrm{ Add}}
\newcommand{\add}{\mathrm{ add}}

\newcommand{\CoDef}{\mathrm{CoDef}}
\newcommand{\Rej}{\mathrm{Rej}}

\newcommand{\Modr}{\mathrm{ Mod}\text{-}}
\newcommand{\Mod}{\text{-}\mathrm{ Mod}}
\newcommand{\modr}{\mathrm{ mod}\text{-}}
\newcommand{\lmod}{\text{-}\mathrm{ mod}}

\begin{document}

\title{Finitely Cosilting Modules}

\author{Flaviu Pop}

\address{Flaviu Pop: "Babe\c s-Bolyai" University, Faculty of Economics and Business Administration, str. T. Mihali, nr. 58-60, 400591, Cluj-Napoca, Romania}

\email{flaviu.v@gmail.com; flaviu.pop@econ.ubbcluj.ro}

%\date{\today}



\begin{abstract} The notion of cosilting module was recently introduced as a generalization of the concept of cotilting module. In this paper, it is introduced the notion of finitely cosilting module, i.e. a cosilting module with some finiteness conditions, and is given a Cosilting Theorem induced by such a module.
\end{abstract}

\maketitle

\section{Introduction}\label{section_introduction} Silting theory is very important in the study of the representation theory. Silting objects in triangulated categories, introduced by Keller and Vossieck in \cite{Keller_Vossieck_1998} in order to study $t$-structures in the bounded derived category of representations of Dynkin quivers, are an important tool in the study of homotopy or derived categories. In \cite{Angeleri_Marks_Vitoria_2016}, it is introduced the notion of (partial) silting module that generalize the notion of tilting module over an arbitrary ring as well as the notion of support $\tau$-tilting module over a finite dimensional algebra (see \cite{AdachiIyamaReiten_2014}), and the authors show that these modules generate torsion classes that provide left approximations and that every partial silting module admits an analogue of the Bongartz complement. Moreover, they connect these modules with silting complexes, $t$-structures and co-$t$-structures in the derived module category. In \cite{AngeleriMarksVitoria_2016_epi}, the authors associate to every partial silting module a ring epimorphism which describe it explicitly as an idempotent quotient of the endomorphism ring of the Bongartz completion. Moreover, it is approached the particular case of hereditary rings. In \cite{Angeleri_Hrbek_2017}, Angeleri-Hugel and Hrbek, give a classification of silting modules over commutative rings, establishing a bijective correspondence with Gabriel filters of finite type. Breaz and Modoi study equivalences induced by silting objects in the (derived) module categories giving a Silting Theorem, this theorem extending the classical Tilting Theorem (see \cite{BreazModoi_2017_1} and \cite{BreazModoi_2017_2}).

The dual notion, of cositing module, was recently introduced in \cite{Breaz_Pop_2017} and \cite{Zhang_Wei_2017_1} as a generalization of the concept of cotilting module. In these papers are proved some characterizations and some important properties of this modules (for instance, the property that all cosilting modules are pure-injective). In \cite{Pop_2017} are characterized cosilting modules in terms of (two-term) cosilting complexes. Cosilting complexes were firstly introduced by Zhang and Wei in \cite{Zhang_Wei_2017_1} as a generalization of the concept of cotilting complex and also as dual of the concept of silting complex (i.e. semi-tilting complex, see \cite{Wei_2013}). A generalization of the notion of cosilting complex, introduced by Zhang and Wei, could be found in \cite{Pop_2017}.

These topics are intensively studied by many other authors (see, for example, \cite{Angeleri_preprint_2016}, \cite{BreazZemlicka}, \cite{MarksStovicek} and \cite{NicolasSaorinZvonareva_2016}).

Zhang and Wei introduce in \cite{Zhang_Wei_2016} and \cite{Zhang_Wei_2017_2} the notion of (hereditary) quasi-cotilting modules and they give a Quasi-Cotilting Theorem, as a generalization of the classical Cotilting Theorem. In \cite{Angeleri_2000_2}, Angeleri-Hugel studies cotilting modules with some finitness conditions defining the notion of finitely cotilting module. She studies some important properties of these modules and she gives a Cotilting Theorem induced by a finitely cotilting module. Following the idea used by Angeleri-Hugel in \cite{Angeleri_2000_2}, in this paper we approach the notion of cosilting module with some finitness conditions defining the concept of finitely cosilting (bi)module and we give a Cosilting Theorem induced by a finitely cosilting bimodule.

The paper is organized as follows. After the introduction which is given in Section \ref{section_introduction}, in Section \ref{section_preliminaries} are presented the notations and the conventions used throughout the paper and they are stated some results which are used in the next sections. In this section is also given a new characterization of cosilting modules (see Proposition \ref{characterization_cosilting}). In Section \ref{section_finitely_cosilting_modules} it is introduced the concept of finitely cosilting module and they are proved some important properties of this new notion. In Section \ref{section_finitely_cosilting_bimodules} it is defined the notion of finitely cosilting bimodule and it is given a Cosilting Theorem (see Theorem \ref{cosilting_theorem}), which is the most important result of the paper.

\section{Preliminaries}\label{section_preliminaries}
Throughout the paper, by a ring $R$ we will understand a unital associative ring, an $R$-module is a right $R$-module and we denote by $\Modr R$ (respectively, $R \Mod$) the category of all right (respectively, left) $R$-modules. We also denote by $\modr R$ (respectively, by $R \lmod$) the subcategory of $\Modr R$ (respectively, of $R \Mod$) consisting of finitely generated $R$-modules. For an $R$-module $T$, we denote by $\Add(T)$ (respectively, by $\add(T)$) the class of all $R$-modules which are isomorphic to direct summands of direct sums (respectively, of finite direct sums) of copies of $T$. Moreover, we consider the perpendicular class ${^{\perp}T}$ defined as ${^{\perp}T} = \{ X \in \Modr R \mid \Ext^{1}_{R}(X,T) = 0 \}$. If $X$ and $T$ are two $R$-modules, then we say that $X$ is cogenerated by $T$ if there is a monomorphism $0 \to X \overset{}\longrightarrow T^{I}$, for some set $I$, and we denote by $\Cogen(T)$ the class of all $T$-cogenerated $R$-modules.

If $R$ and $S$ are two rings and $T$ is an $(S,R)$-bimodule, then the $\Hom$'s contravariant functors induced by $T$, i.e. $\Delta_{R}(-) = \Hom_{R}(-,T) : \Modr R \rightleftarrows S \Mod : \Hom_{S}(-,T) = \Delta_{S}(-),$ are adjoint on the right. When the ring $R$ or $S$ are understood, we simply write $\Delta$. The natural transformations associated to this adjunction, $\delta : 1 \to \Delta^{2}$, are in fact the evaluation maps, $\delta_X : X \to \Delta^{2}(X), \text{   } \delta_{X}(x) : f \mapsto f(x),$ and they satisfy the identities $\Delta(\delta_{X}) \circ \delta_{\Delta(X)} = 1_{\Delta(X)}$, for all modules $X$. A module $X$ is called $\delta$-torsionless (respectively, $\delta$-reflexive) if $\delta_{X}$ is a monomorphism (respectively, an isomorphism). We denote by $\Refl_{\delta}$ the class of all $\delta$-reflexive modules. From the identities above, we observe that $\Delta(X)$ is $\delta$-torsionless, for any module $X$. If $X$ is a module and if we denote the reject of $X$ in $T$ by $\Rej_{T}(X)$, i.e. $\Rej_{T}(X) = \bigcap_{f\in\Hom_{R}(X,T)}\Ker(f)$, we have that $\Ker(\delta_X) = \Rej_{T}(X)$, hence $X$ is $\delta$-torsionless if and only if $X$ is $T$-cogenerated.

In order to define the notion of {\sl cosilting module}, we have to define the class $\CB_{\zeta}$, as follows. If $\zeta : Q_{0} \to Q_{1}$ is an $R$-homomorphism, then the class $\CB_{\zeta}$ is defined as $$\CB_{\zeta} = \{X \in \Modr R \mid \Hom_{R}(X,\zeta) \text{ is an epimorphism} \}.$$

\begin{remark} Let $\zeta : Q_{0} \to Q_{1}$ be an $R$-homomorphism. We define {\sl the codefect functor} as follows. For an object $X \in \Modr R$, we set $\CoDef_{\zeta}(X) = \Coker\Hom_{R}(X,\zeta)$ and for a morphism $f : X \to Y$ in $\Modr R$, we define $\CoDef_{\zeta}(f) : \CoDef_{\zeta}(Y) \to \CoDef_{\zeta}(X)$ by $\CoDef_{\zeta}(f) = \phi,$ where $\phi$ is given by the universal property of the cokernel, as we see in the following diagram.

\[\xymatrixcolsep{3pc}\xymatrixrowsep{3pc}
\xymatrix{
&\Hom_{R}(Y,Q_{0})\ar[r]^{\Hom_{R}(Y,\zeta)}\ar[d]_{\Hom_{R}(f,Q_{0})} &\Hom_{R}(Y,Q_{1})\ar[r]^{\pi_{Y}}\ar[d]^{\Hom_{R}(f,Q_{1})} &\Coker\Hom_{R}(Y,\zeta)\ar[r]\ar[d]^{\phi}&0\\
&\Hom_{R}(X,Q_{0})\ar[r]^{\Hom_{R}(X,\zeta)} &\Hom_{R}(X,Q_{1})\ar[r]^{\pi_{X}} &\Coker\Hom_{R}(X,\zeta)\ar[r]&0\\
}
\]

We observe that the class $\CB_{\zeta}$ is in fact the kernel of $\CoDef_{\zeta}$, i.e. $\Ker\CoDef_{\zeta} = \CB_{\zeta}$.
\end{remark}

The following lemma give some important (closure) properties associated to the class $\CB_{\zeta}$ and is very used throughout the paper.

\begin{lemma}\label{closure_prop}\cite[Lemma 2.3]{Breaz_Pop_2017} Let $\zeta: Q_{0} \to Q_{1}$ be an $R$-homomorphism between injective $R$-modules and let $T=\Ker(\zeta)$. Then the following statements hold:
\begin{enumerate}[{\rm (1)}]
   \item The class $\CB_{\zeta}$ is closed under direct sums, submodules and extensions.
   \item $\CB_{\zeta} \subseteq {^{\perp}T}$.
   \item If $0 \to A \overset{f}\longrightarrow B \overset{g}\longrightarrow X \to 0$ is an exact sequence such that $A$ and $B$ belong to the class $\CB_{\zeta}$ and $X \in {^{\perp}T}$ then $X\in\CB_{\zeta}$.
\end{enumerate}
\end{lemma}

Now, we give the definition of {\sl cosilting module}, introduced independently in \cite{Breaz_Pop_2017} and \cite{Zhang_Wei_2017_1}, dualizing the notion of {\sl silting module} introduced in \cite{Angeleri_Marks_Vitoria_2016} and, at the same time, generalizing the notion of {\sl cotilting module}.

\begin{definition} We say that an $R$-module $T$ is:
\begin{enumerate}[{(a)}]
   \item {\sl partial cosilting} ({\sl with respect to $\zeta$}), if there exists an injective copresentation of $T$ $$0 \to T \overset{f}\longrightarrow Q_{0} \overset{\zeta}\longrightarrow Q_{1}$$such that:
       \begin{enumerate}[(a)]
           \item[(1)] $T\in\CB_{\zeta}$, and
           \item[(2)] the class $\CB_{\zeta}$ is closed under direct products;
       \end{enumerate}
   \item {\sl cosilting} ({\sl with respect to $\zeta$}), if there exists an injective copresentation $$0 \to T \overset{f}\longrightarrow Q_{0} \overset{\zeta}\longrightarrow Q_{1}$$ of $T$ such that $\Cogen(T)=\CB_{\zeta}$.
\end{enumerate}
\end{definition}

We have the following characterization of cosilting modules.

\begin{proposition}\label{characterization_cosilting} Let $T$ be an $R$-module and let $T=\Ker(\zeta)$, where $\zeta:Q_{0} \to Q_{1}$ is a homomorphism between injective $R$-modules. Then $T$ is cosilting with respect to $\zeta$ if and only if
\begin{itemize}
   \item[(i)] $T^{I} \in \Ker\CoDef_{\zeta}$, for all sets $I$;
   \item[(ii)] $\Ker\Hom_{R}(-,T)\cap\Ker\CoDef_{\zeta}=0$.
\end{itemize}
\end{proposition}
\begin{proof} Suppose that $T$ is cosilting with respect to $\zeta$. Then we have the equality $\Cogen(T) = \CB_{\zeta}$. It follows that $T^{I} \in \Ker\CoDef_{\zeta}$, for all sets $I$. Now let $X \in \Ker\Hom_{R}(-,T) \cap \Ker\CoDef_{\zeta}$, i.e. $\Hom_{R}(X,T) = 0$ and $X \in \CB_{\zeta}$. Since $X \in \Cogen(T)$, there is a monomorphism $0 \to X \overset{f}\longrightarrow T^{I}$. If we assume that $X \ne 0$, then $f \ne 0$, hence $\Hom_{R}(X,T^{I}) \ne 0$, so that $\Hom_{R}(X,T)^{I} \ne 0$ and then $\Hom_{R}(X,T) \ne 0$, which is a contradiction. It follows that $X = 0$.

Conversely, suppose that the statements (i) and (ii) hold. Let $X \in \Cogen(T)$. Then there is a monomorphism $0 \to X \overset{f}\longrightarrow T^{I}$. By (i), we have that $T^{I}\in\CB_{\zeta}$ and, by Lemma \ref{closure_prop}(1), it follows that $X \in \CB_{\zeta}$.

In order to show the reverse inclusion, we consider $X\in\CB_{\zeta}$. Then we have the exact sequence $0 \to \Rej_{T}(X) \overset{i}\longrightarrow X\overset{p}\longrightarrow X/\Rej_{T}(X) \to 0$. Since, in general, $X/\Rej_{T}(X) \in \Cogen(T)$, it follows from above that $X/\Rej_{T}(X) \in \CB_{\zeta}$. By Lemma \ref{closure_prop}(2), we obtain $X/\Rej_{T}(X) \in {^{\perp}T}$, i.e. $\Ext^{1}_{R}(X/\Rej_{T}(X),T) = 0$, hence we have the following short exact sequence $0 \to \Hom_{R}(X/\Rej_{T}(X),T) \overset{\Hom_{R}(p,T)}\longrightarrow \Hom_{R}(X,T) \overset{\Hom_{R}(i,T)}\longrightarrow \Hom_{R}(\Rej_{T}(X),T) \to 0$. Since $\Hom_{R}(p,T)$ is, in general, an isomorphism, we obtain that $\Hom_{R}(\Rej_{T}(X),T)=0$, hence $\Rej_{T}(X) \in \Ker\Hom_{R}(-,T)$. On the other hand, applying Lemma \ref{closure_prop}(1), we obtain that $\Rej_{T}(X) \in \CB_{\zeta} = \Ker\CoDef_{\zeta}$. From (ii), we have that $\Rej_{T}(X) = 0$, hence $X \in \Cogen(T)$.
\end{proof}

Let $\CM$ be a class in $\Modr R$ and let $X$ be a right $R$-module. We say that an $R$-homomorphism $f : X \to M$ is an $\CM$-preenvelope of $X$ if $M \in \CM$ and the abelian group homomorphism $\Hom_{R}(f,M^{\prime}) : \Hom_{R}(M,M^{\prime}) \to \Hom_{R}(X,M^{\prime})$ is an epimorphism, for all $M^{\prime} \in \CM$.

\begin{proposition}\label{Angeleri_Corollary_1.5}\cite[Corollary 1.5]{Angeleri_2000_1} Let $X$ and $T$ be two right $R$-modules and suppose that $S = \End_{R}(T)$. Then $X_{R}$ has an $\add(T_{R})$-preenvelope if and only if the left $S$-module $_{S}\Hom_{R}(X,T)$ is finitely generated.
\end{proposition}

Next, we recall the definition of ($\pi$-)coherent module (see, for instance, \cite{Angeleri_2000_1}) and then we state a very useful result related to this notion.

\begin{definition}Let $X$ be an $R$-module. Then $X$ in called:
\begin{itemize}
   \item[(a)] {\sl coherent} if:
      \begin{itemize}
         \item[(1)] $X$ is finitely presented;
         \item[(2)] every finitely generated submodule of $X$ is finitely presented.
      \end{itemize}
   \item[(b)] {\sl $\pi$-coherent} if:
      \begin{itemize}
         \item[(1)] $X$ is finitely presented;
         \item[(2)] every finitely generated $R$-module, which is cogenerated by $X$, is finitely presented.
      \end{itemize}
\end{itemize}
We say that the ring $R$ is a {\sl right} (respectively, {\sl left}) {\sl ($\pi$-)coherent ring} if the $R$-module $R_{R}$ (respectively, $_{R}R$) is a ($\pi$-)coherent $R$-module.
\end{definition}

\begin{theorem}\label{Angeleri_Theorem_3.7}\cite[Theorem 3.17]{Angeleri_2000_1} Let $T$ be a finitely generated right $R$-module with $S = \End_{R}(T)$. If every finitely generated module has an $\add(T_{R})$-preenvelope, then the left $S$-module $_{S}T$ is $\pi$-coherent.
\end{theorem}

\section{Finitely cosilting modules} \label{section_finitely_cosilting_modules}

For a right $R$-module $T_{R}$, we consider the following classes $\CY_{R} = \Cogen(T_{R}) \cap \modr R$ and $\CX_{R} = \Ker\Delta_{R} \cap \modr R$. Moreover, we denote $\Ext_{R}^{1}(-,T)$ by $\Gamma_{R}(-)$ and, if the ring $R$ is understood, we simply write $\Gamma(-)$.

\begin{lemma}\label{Angeleri_Lemma_2.5} Let $T_{R}$ be an $R$-module such that $T=\Ker(\zeta)$, where $\zeta : Q_{0} \to Q_{1}$ is an $R$-homomorphism between injective $R$-modules. Suppose that all $R$-modules in $\CY_{R}$ are finitely presented and belong to the class $\CB_{\zeta}$. Let $X$ be a finitely generated $R$-module and let $X^{\prime} = \Rej_{T}(X)$. Then the following statements hold.
\begin{itemize}
   \item[(1)] $X^{\prime}\in\CX_{R}$;
   \item[(2)] $X/X^{\prime}\in\CY_{R}$;
   \item[(3)] $\Delta(X)\cong\Delta(X/X^{\prime})$.
\end{itemize}
Moreover, if $T_{R}$ is faithful then $\Gamma(X)\cong\Gamma(X^{\prime})$.
\end{lemma}
\begin{proof} Since $X/X^{\prime}\in\Cogen(T)$ and since $X \in \modr R$, hence $X/X^{\prime} \in \modr R$, we obtain that $X/X^{\prime}\in\CY_{R}$. By hypothesis, we have that $X/X^{\prime}$ is finitely presented and $X/X^{\prime}\in\CB_{\zeta}$. By Lemma \ref{closure_prop}(2), we have $\Gamma(X/X^{\prime})=0$. Applying the contravariant $\Delta$ functor to the canonical short exact sequence $$0 \to X^{\prime} \overset{i}\longrightarrow X \overset{p}\longrightarrow X/X^{\prime} \to 0$$ we obtain the following two exact sequences $$0 \to \Delta(X/X^{\prime}) \overset{\Delta(p)}\longrightarrow \Delta(X) \overset{\Delta(i)}\longrightarrow \Delta(X^{\prime}) \to 0$$ and $$ 0 \to \Gamma(X) \overset{\Gamma(i)}\longrightarrow \Gamma(X^{\prime}) \overset{\delta^{1}}\longrightarrow \Ext^{2}_{R}(X/X^{\prime},T) \longrightarrow \dots$$Since $\Delta(p)$ is an isomorphism, hence $\Delta(X)\cong\Delta(X/X^{\prime})$, we obtain that $\Delta(X^{\prime})=0$, i.e. $X^{\prime}\in\Ker\Delta$. By \cite[25.4]{Wisbauer_1991}, $X^{\prime}$ is finitely generated, i.e. $X^{\prime} \in \modr R$. Therefore $X^{\prime}\in\CX_{R}$.

Now, suppose that $T_{R}$ is faithful. We show that $\Ext^{2}_{R}(Y,T)=0$, for all finitely presented $R$-modules $Y$. Let $Y$ be a finitely presented $R$-module. Then there is an exact sequence $0 \to K \overset{i}\longrightarrow R^{n} \overset{g}\longrightarrow Y \to 0$ with $K \in \modr R$. Applying the contravariant $\Delta$ functor, we obtain the long exact sequence$$0 \to \Delta(Y) \overset{\Delta(g)}\longrightarrow \Delta(R^{n}) \overset{\Delta(i)}\longrightarrow \Delta(K) \overset{\delta^{0}}\longrightarrow$$ $$\overset{\delta^{0}}\longrightarrow \Gamma(Y) \overset{\Gamma(g)}\longrightarrow \Gamma(R^{n}) \overset{\Gamma(i)}\longrightarrow \Gamma(K) \overset{\delta^{1}}\longrightarrow$$ $$\overset{\delta^{1}}\longrightarrow \Ext^{2}_{R}(Y,T) \overset{\Ext^{2}_{R}(g,T)}\longrightarrow \Ext^{2}_{R}(R^{n},T) \overset{\Ext^{2}_{R}(i,T)}\longrightarrow \Ext^{2}_{R}(K,T) \overset{\delta^{2}}\longrightarrow \dots$$Since $T_{R}$ is faithful, we have that $R\in\Cogen(T_{R})$, hence $R^{n}\in\Cogen(T_{R})$ and thus $K\in\Cogen(T_{R})$. Then $K\in\CY_{R}$ and, by hypothesis, we have that $K\in\CB_{\zeta}$. From Lemma \ref{closure_prop}(2), we have $\Gamma(K)=0$ and since $\Ext^{2}_{R}(R^{n},T)=0$ we obtain that $\Ext^{2}_{R}(Y,T)=0$.

Therefore, since $X/X^{\prime}$ is finitely presented, we have that  $\Ext^{2}_{R}(X/X^{\prime},T) = 0$, hence $\Gamma(X)\cong\Gamma(X^{\prime})$.
\end{proof}

\begin{lemma}\label{Angeleri_Proposition_1.1.2} Let $T_{R}$ be an $R$-module and let $\zeta : Q_{0} \to Q_{1}$ be a homomorphism between injective $R$-modules. Assume that $T=\Ker(\zeta)$ and $\Ker\Hom_{R}(-,T) \cap \Ker\CoDef_{\zeta} = 0$. If $X$ is a right $R$-module such that $X\in\CB_{\zeta}$ and $X/\Rej_{T}(X)\in\CB_{\zeta}$, then $X\in\Cogen(T_{R})$.
\end{lemma}
\begin{proof} Let us denote $\Rej_{T}(X)$ by $X^{\prime}$. Since $X/X^{\prime} \in \CB_{\zeta}$, it follows, by Lemma \ref{closure_prop}(2), that $\Gamma(X/X^{\prime}) = 0$. Applying the contravariant $\Delta$ functor to the short exact sequence $0 \to X^{\prime} \overset{i}\longrightarrow X \overset{p}\longrightarrow X/X^{\prime} \to 0$, we obtain the short exact sequence $0 \to \Delta(X/X^{\prime}) \overset{\Delta(p)}\longrightarrow \Delta(X) \overset{\Delta(i)}\longrightarrow \Delta(X^{\prime}) \to 0$ and then, taking into account that $\Delta(p)$ is an isomorphism, we have that $\Delta(X^{\prime}) = 0$, i.e. $X^{\prime}\in\Ker\Delta$. By Lemma \ref{closure_prop}(1), we have that $X^{\prime} \in \Ker\CoDef_{\zeta}$. From hypothesis, we have that $X^{\prime}=0$, hence $X\in\Cogen(T_{R})$.
\end{proof}

\begin{proposition}\label{Angeleri_Proposition_1.1.1} Let $T_{R}$ be an $R$-module with $S=\End_{R}(T)$ and let $\zeta : Q_{0} \to Q_{1}$ be a homomorphism between injective $R$-modules. Assume that $T=\Ker(\zeta)$ and $T \in \Ker\CoDef_{\zeta}$. Let $X\in\Cogen(T_{R})$. Then $\Delta_{}(X)$ is finitely generated left $S$-module if and only if there is an exact sequence $0 \to X \overset{}\longrightarrow T^{n} \overset{}\longrightarrow Z \to 0$ with $Z\in\CB_{\zeta}$.
\end{proposition}
\begin{proof} Suppose that $\Delta_{}(X)$ is finitely generated left $S$-module. By \cite[Proposition 4.2.2]{Colby_Fuller_2004}, there is a monomorphism $0 \to X \overset{f}\longrightarrow T^{n}$ such that $\Delta_{}(f)$ is an epimorphism. Hence we have an exact sequence $0 \to X \overset{f}\longrightarrow T^{n} \overset{p}\longrightarrow Z \to 0$ and applying the contravariant $\Delta_{}$ functor, we obtain the long exact sequence $$0 \to \Delta(Z) \overset{\Delta(p)}\longrightarrow \Delta(T^{n}) \overset{\Delta(f)}\longrightarrow \Delta(X) \overset{\delta^{0}}\longrightarrow$$ $$\overset{\delta^{0}}\longrightarrow \Gamma(Z) \overset{\Gamma(p)}\longrightarrow \Gamma(T^{n}) \overset{\Gamma(f)}\longrightarrow \Gamma(X) \overset{\delta^{1}}\longrightarrow \dots$$
Since $T\in\CB_{\zeta}$, we have, by Lemma \ref{closure_prop}(1), that $T^{n} \in \CB_{\zeta}$ and $X \in \CB_{\zeta}$. Thus $\Gamma(T^{n}) = 0$ and taking into account that $\Delta_{}(f)$ is an epimorphism, we obtain that $\Gamma(Z) = 0$, i.e. $Z \in {^{\perp}}T$. By Lemma \ref{closure_prop}(3), we get that $Z\in\CB_{\zeta}$.

Conversely, suppose that there is an exact sequence $0 \to X \overset{f}\longrightarrow T^{n} \overset{g}\longrightarrow Z \to 0$ with $Z\in\CB_{\zeta}$. Hence we have the following long exact sequence induced by $\Delta_{}$ $$0 \to \Delta(Z) \overset{\Delta(g)}\longrightarrow \Delta(T^{n}) \overset{\Delta(f)}\longrightarrow \Delta(X) \overset{\delta^{0}}\longrightarrow$$ $$\overset{\delta^{0}}\longrightarrow \Gamma(Z) \overset{\Gamma(g)}\longrightarrow \Gamma(T^{n}) \overset{\Gamma(f)}\longrightarrow \Gamma(X) \overset{\delta^{1}}\longrightarrow \dots$$ Since $Z\in\CB_{\zeta}$, we have $\Gamma(Z)=0$, hence $S^{n} \cong \Delta(T^{n}) \overset{\Delta(f)}\longrightarrow \Delta(X) \to 0$ is an epimorphism, i.e. $\Delta(X)$ is finitely generated left $S$-module.
\end{proof}

\begin{proposition}\label{Angeleri_Proposition_1.2.2} Let $T_{R}$ be a finitely generated $R$-module with $S=\End_{R}(T)$. Suppose that $\Delta_{R} : \Modr R \to S \Mod$ carries finitely generated modules to finitely generated modules. Then the following assertions hold.
\begin{itemize}
   \item[(1)] If $_{S}Y \in \Cogen(_{S}T)$ is finitely generated, then $_{S}Y$ is finitely presented.
   \item[(2)] $\Delta_{R}$ carries finitely generated modules to finitely presented modules.
   \item[(3)] $\Gamma_{R}$ carries finitely presented modules to finitely presented modules.
\end{itemize}
\end{proposition}
\begin{proof} By Proposition \ref{Angeleri_Corollary_1.5} and Theorem \ref{Angeleri_Theorem_3.7}, we have that $_{S}T$ is $\pi$-coherent. Thus (1) follows.

If $X$ is a finitely generated $R$-module, then $\Delta_{R}(X)$ is also finitely generated $S$-module (by hypothesis). But $\Delta_{R}(X) \in \Cogen(_{S}T)$ and then (2) follows by applying (1).

Let $X$ be a finitely presented $R$-module. Then there is an exact sequence $0 \to K \overset{i}\longrightarrow R^{m} \overset{g}\longrightarrow X \to 0$ with $K$ being finitely generated $R$-module. Applying the contravariant $\Delta_{R}$ functor to this sequence, we obtain the induced long exact sequence $$0 \to \Delta_{R}(X) \overset{\Delta_{R}(g)}\longrightarrow \Delta_{R}(R^{m}) \overset{\Delta_{R}(i)}\longrightarrow \Delta_{R}(K) \overset{\delta^{0}}\longrightarrow$$ $$\overset{\delta^{0}}\longrightarrow \Gamma_{R}(X) \overset{\Gamma_{R}(g)}\longrightarrow \Gamma_{R}(R^{m}) \overset{\Gamma_{R}(i)}\longrightarrow \Gamma_{R}(K) \overset{\delta^{1}}\longrightarrow \dots$$

Since $\Gamma_{R}(R^{m}) = 0$, we have the following short exact sequence $0 \to \Ker(\delta^{0}) \overset{j}\longrightarrow \Delta_{R}(K) \overset{\delta^{0}}\longrightarrow \Gamma_{R}(X) \to 0$. From (2), the left $S$-modules $\Delta_{R}(R^{m})$ and $\Delta_{R}(K)$ are finitely presented. It follows that $\Ker(\delta^{0})$ is finitely generated. Then (3) follows by \cite[25.1]{Wisbauer_1991}.
\end{proof}

Now, we define the notion of finitely cosilting module.

\begin{definition} Let $T$ be a right $R$-module with $S=\End_{R}(T)$ and let $T=\Ker(\zeta)$, where $\zeta : Q_{0} \to Q_{1}$ is an $R$-homomorphism between injective $R$-modules. We say that $T$ is {\sl finitely cosilting with respect to $\zeta$} if the following conditions are satisfied:
\begin{itemize}
   \item[(1)] $T\in\Ker\CoDef_{\zeta}$;
   \item[(2)] $\Ker\Hom_{R}(-,T)\cap\Ker\CoDef_{\zeta}=0$;
   \item[(3)] $T$ is finitely generated;
   \item[(4)] $\Hom_{R}(-,T):\Modr R \to S \Mod$ carries finitely generated modules to finitely generated modules.
\end{itemize}
\end{definition}

\begin{proposition}\label{Angeleri_Proposition_1.1.3} If $T_{R}$ is a finitely cosilting $R$-module with respect to the injective copresentation $\zeta : Q_{0} \to Q_{1}$, then $\CY_{R} = \CB_{\zeta} \cap \modr R$.
\end{proposition}
\begin{proof} Let $X\in\CY_{R}$. Then $X \in \Cogen(T_{R})$ and $X \in \modr R$. Since $T_{R}$ is finitely cosilting with respect to $\zeta$, we have that $T \in \CB_{\zeta}$ and $\Delta(X)$ is finitely generated left $S$-module. By Proposition \ref{Angeleri_Proposition_1.1.1}, there is an exact sequence $0 \to X \overset{}\longrightarrow T^{n} \overset{}\longrightarrow Z \to 0$ with $Z \in \CB_{\zeta}$. From Lemma \ref{closure_prop}(1), we obtain that $X\in\CB_{\zeta}$.

For the reverse inclusion, let consider $X \in \CB_{\zeta} \cap \modr R$. In general, we have $X/\Rej_{T}(X) \in \Cogen(T_{R})$ and $\Delta(X/\Rej_{T}(X)) \cong \Delta(X)$. Since $T_{R}$ is finitely cosilting with respect to $\zeta$, we have that $T \in \CB_{\zeta}$ and $\Delta(X) \in S \lmod$, hence we have $\Delta(X/\Rej_{T}(X)) \in S\lmod$. By Proposition \ref{Angeleri_Proposition_1.1.1}, there is an exact sequence $0 \to X/\Rej_{T}(X) \overset{f}\longrightarrow T^{n} \overset{g}\longrightarrow Z \to 0$ with $Z \in \CB_{\zeta}$. From Lemma \ref{closure_prop}(1), $X/\Rej_{T}(X) \in \CB_{\zeta}$, hence, by Lemma \ref{Angeleri_Proposition_1.1.2}, we have $X\in\Cogen(T_{R})$.
\end{proof}

\begin{proposition}\label{Angeleri_Proposition_1.2.1} If $T_{R}$ is a finitely cosilting $R$-module with respect to the injective copresentation $\zeta : Q_{0} \to Q_{1}$, then $\CY_{R} \subseteq \Refl_{\delta}$.
\end{proposition}
\begin{proof} Let $X \in \CY_{R}$. Then $X \in \Cogen(T_{R})$ and $X \in \modr R$. Since $T$ is finitely cosilting with respect to $\zeta$, we have that $\Delta(X) \in S\lmod$, where $S=\End_{R}(T)$. By Proposition \ref{Angeleri_Proposition_1.1.1}, there is an exact sequence $0 \to X \overset{f}\longrightarrow T^{n} \overset{g}\longrightarrow Z \to 0$ with $Z \in \CB_{\zeta}$. It follows, by Lemma \ref{closure_prop}(2), that $\Gamma(Z) = 0$, hence the induced sequence $0 \to \Delta(Z) \overset{\Delta(g)}\longrightarrow \Delta(T^{n}) \overset{\Delta(f)}\longrightarrow \Delta(X) \to 0$ is also exact. Now, consider the following commutative diagram with exact rows

\[\xymatrixcolsep{3pc}\xymatrixrowsep{3pc}
\xymatrix{
&0\ar[r]^{} &X\ar[r]^{f}\ar[d]_{\delta_{X}} &T^{n}\ar[r]^{g}\ar[d]^{\delta_{T^{n}}} &Z\ar[r]\ar[d]^{\delta_{Z}} &0\\
&0\ar[r]^{} &\Delta^{2}(X)\ar[r]^{\Delta^{2}(f)} &\Delta^{2}(T^{n})\ar[r]^{\Delta^{2}(g)} &\Delta^{2}(Z)\\
}
\]

Since $T_{R}$ is finitely generated, we have that $Z$ is finitely generated, hence $Z \in \CB_{\zeta} \cap \modr R$. By Proposition \ref{Angeleri_Proposition_1.1.3}, we have $Z \in \CY_{R}$, so that $Z \in \Cogen(T_{R})$ and thus $\delta_{Z}$ is a monomorphism. From the fact that $\delta_{T^{n}}$ is an isomorphism, it follows, by applying the Snake Lemma to the above diagram, that $\delta_{X}$ is an isomorphism.

\end{proof}

\begin{proposition}\label{Angeleri_Proposition_2.6_1} Let $T_{R}$ be a finitely presented $R$-module. If $R$ is a right coherent ring and $T_{R}$ is finitely cosilting with respect to the injective copresentation $\zeta : Q_{0} \to Q_{1}$, then the class $\CY_{R}$ consists of the finitely presented modules belonging to the class $\CB_{\zeta}$.
\end{proposition}
\begin{proof} Let $X \in \CY_{R}$. Then $X\in\Cogen(T_{R})$ and $X\in\modr R$. Since $T_{R}$ is finitely cosilting with respect to $\zeta$, we have $T \in \CB_{\zeta}$ and $\Delta_{R}(X)$ is finitely generated left $S$-module, where $S=\End_{R}(T)$. It follows, by Proposition \ref{Angeleri_Proposition_1.1.1}, that there is an exact sequence $0 \to X \overset{f}\longrightarrow T^{n} \overset{g}\longrightarrow Z \to 0$ with $Z \in \CB_{\zeta}$. From the facts that $T^{n}$ is finitely presented, because $T$ is finitely presented, and $R$ is a right coherent ring, we obtain that $T^{n}$ is coherent in $\Modr R$ (by \cite[26.6]{Wisbauer_1991}). Thus $X$ is finitely presented. Since $T\in\CB_{\zeta}$, it follows, by Lemma \ref{closure_prop}(1), that $X\in\CB_{\zeta}$.

Let $X$ be a finitely presented $R$-module such that $X \in \CB_{\zeta}$. If we denote $\Rej_{T}(X)$ by $X^{\prime}$, we have, by Lemma \ref{Angeleri_Lemma_2.5}, that $X^{\prime} \in \Ker\Delta_{R}$. Since $X \in \CB_{\zeta}$, we obtain, by Lemma \ref{closure_prop}(1), that $X^{\prime} \in \CB_{\zeta}$. Since $T_{R}$ is finitely cosilting, it follows that $X^{\prime} = 0$, hence $X\in\Cogen(T_{R})$. Therefore $X \in \Cogen(T_{R}) \cap \modr R$, i.e. $X\in\CY_{R}$.
\end{proof}

\section{Finitely cosilting bimodules and a Cosilting Theorem }\label{section_finitely_cosilting_bimodules}

We start this section with the definition of finitely cosilting bimodule.

\begin{definition} Let $_{S}T_{R}$ be an $(S,R)$-bimodule. Then $T$ is called {\sl finitely cosilting bimodule} if:
\begin{itemize}
   \item[(1)] $_{S}T_{R}$ is a faithfully balanced bimodule;
   \item[(2)] $T_{R}$ and $_{S}T$ are finitely cosilting modules (with respect to the injective copresentations $\zeta : Q_{0} \to Q_{1}$ and $\tau : S_{0} \to S_{1}$, respectively).
\end{itemize}
\end{definition}

\begin{lemma}\label{fg_cogen_is_fp} Let $_{S}T_{R}$ be a finitely cosilting bimodule. Then every finitely generated module $X_{R} \in \Cogen(T_{R})$ is finitely presented.
\end{lemma}
\begin{proof} First, we show that the module $T_{R}$ is $\pi$-coherent. Let $_{S}Y \in S\lmod$, i.e. $_{S}Y$ is finitely generated left $S$-module. Since $_{S}T$ is finitely cosilting module, we have that $\Hom_{S}(Y,T)$ is finitely generated right $R$-module and then, by Proposition \ref{Angeleri_Corollary_1.5}, $_{S}Y$ has an $\add(_{S}T)$-preenvelope. Since $_{S}T$ is finitely generated, it follows, by Theorem \ref{Angeleri_Theorem_3.7}, that $T_{R}$ is $\pi$-coherent.

Now, if we consider $X_{R} \in \Cogen(T_{R})$ to be a finitely generated $R$-module, we observe that $X_{R}$ is finitely presented.
\end{proof}

\begin{corollary}\label{Angeleri_Prposition_4.3_1} Let $_{S}T_{R}$ be a finitely cosilting bimodule. Then $T_{R}$ is finitely presented and $R$ is a right coherent ring.
\end{corollary}
\begin{proof} We have that $T_{R}$ is finitely presented by Lemma \ref{fg_cogen_is_fp}.

Let $I$ be a right ideal of $R$ which is finitely generated. Since $T_{R}$ is faithful, we have $R_{R} \in \Cogen(T_{R})$, hence $I_{R} \in \Cogen(T_{R})$. Since $I_{R}$ is finitely generated, it follows, by Lemma \ref{fg_cogen_is_fp}, that $I_{R}$ is finitely presented. Therefore $R$ is a right coherent ring.
\end{proof}

\begin{proposition}\label{rezultat_1} Let $_{S}T_{R}$ be a finitely cosilting bimodule. Then the following assertions hold.
\begin{itemize}
   \item[(1)] The pair $(\CX_{R}, \CY_{R})$ is a torsion pair in $\modr R$.
   \item[(2)] The pair $({_{S}\CX}, {_{S}\CY})$ is a torsion pair in $S \lmod$.
\end{itemize}
\end{proposition}
\begin{proof}(1) First, we prove that, for every $X \in \modr R$, we have $\Hom_{R}(X,Y)=0$, for all $Y \in \CY_{R}$, if and only if $X \in \CX_{R}$.

Let $X\in\modr R$. If $\Hom_{R}(X,Y)=0$, for all $Y \in \CY_{R}$, it follows that $\Hom_{R}(X,T)=0$ (since $T_{R} \in \CY_{R}$) and thus $X\in\CX_{R}$. Conversely, assume that $X\in\CX_{R}$. Then $X\in\Ker\Delta_{R}$, i.e. $\Hom_{R}(X,T)=0$. If $Y \in \CY_{R}$, then $Y \in \Cogen(T_{R})$, hence there is a monomorphism $0 \to Y \overset{f}\longrightarrow T^{I}$, for some set $I$. Applying the covariant $\Hom_{R}(X,-)$ functor, we obtain that $0 \to \Hom_{R}(X,Y) \overset{\Hom_{R}(X,f)}\longrightarrow \Hom_{R}(X,T^{I})$ is a monomorphism. Then $\Hom_{R}(X,Y)=0$.

Second, we prove that, for every $Y \in \modr R$, we have $\Hom_{R}(X,Y)=0$, for all $X \in \CX_{R}$, if and only if $Y \in \CY_{R}$.

Suppose that $T_{R}$ is finitely cosilting with respect to the injective copresentation $\zeta : Q_{0} \to Q_{1}$. Since $_{S}T_{R}$ is finitely cosilting bimodule, it follows, by Corollary \ref{Angeleri_Prposition_4.3_1}, that $R$ is a right coherent ring and $T_{R}$ is finitely presented module. By Proposition \ref{Angeleri_Proposition_2.6_1}, we have that the class $\CY_{R}$ consists of finitely presented modules belonging to the class $\CB_{\zeta}$.

Let $Y \in \modr R$. Assume that $\Hom_{R}(X,Y)=0$, for all $X \in \CX_{R}$. By Lemma \ref{Angeleri_Lemma_2.5}, $Y^{\prime} \in \CX_{R}$, where $Y^{\prime} = \Rej_{T}(Y)$. Then $\Hom_{R}(Y^{\prime},Y) = 0$, hence $Y^{\prime}=0$ (because $Y^{\prime}$ is a submodule of $Y$) and thus $Y \in \Cogen(T_{R})$. Therefore $Y \in \CY_{R}$. Conversely, suppose that $Y\in\CY_{R}$. Then $Y\in\Cogen(T_{R})$, hence there is a monomorphism $0 \to Y \overset{f}\longrightarrow T^{I}$. Consider $X \in \CX_{R}$. Then $\Hom_{R}(X,T)=0$. Since the covariant $\Hom_{R}(X,-)$ functor is left exact, we have that $0 \to \Hom_{R}(X,Y) \overset{\Hom_{R}(X,f)}\longrightarrow \Hom_{R}(X,T^{I})$ is a monomorphism. Thus $\Hom_{R}(X,Y)=0$.

(2) It is proved similarly.
\end{proof}

\begin{proposition}\label{rezultat_0} Let $_{S}T_{R}$ be a finitely cosilting bimodule. Then $$\Delta_{R} : \CY_{R} \rightleftarrows {_{S}\CY} : \Delta_{S}$$ is a duality.
\end{proposition}
\begin{proof} By Proposition \ref{Angeleri_Proposition_1.2.1}, we have that $\CY_{R} \subseteq \Refl_{\delta}$ and similarly we can prove that ${_{S}\CY} \subseteq \Refl_{\delta}$.

Let $X \in \CY_{R}$. Then $X \in \Cogen(T_{R})$ and $X \in \modr R$. Since $T_{R}$ is a finitely cosilting module, we have $\Delta_{R}(X) \in S \lmod$. But, in general, we have that $\Delta_{R}(X) \in \Cogen(_{S}T)$ and thus $\Delta_{R}(X) \in {_{S}\CY}$. Therefore $\Delta_{R}(\CY_{R}) \subseteq {_{S}\CY}$. Similarly, we have $\Delta_{S}({_{S}\CY}) \subseteq \CY_{R}$.
\end{proof}

If $R$ is a ring then we denote by $\presr R$ (respectively, by $R \presl$) the subcategory of $\Modr R$ (respectively, of $R \Mod$) consisting of finitely presented $R$-modules. Moreover, given a bimodule $_{S}T_{R}$, we consider the following classes $\CY^{\prime}_{R} = \Cogen(T_{R}) \cap \presr R$ and ${_{S}\CY^{\prime}} = \Cogen(_{S}T) \cap S \presl$.

\begin{proposition}\label{rezultat_2} If $_{S}T_{R}$ is a finitely cosilting bimodule then $\CY_{R} = \CY_{R}^{\prime}$ and ${_{S}\CY} = {_{S}\CY^{\prime}}$. Consequently, we have that $\Delta_{R} : \CY_{R}^{\prime} \rightleftarrows {_{S}\CY^{\prime}} : \Delta_{S}$ is a duality.
\end{proposition}
\begin{proof} The inclusion $\CY_{R}^{\prime} \subseteq \CY_{R}$ is obvious. By Corollary \ref{Angeleri_Prposition_4.3_1}, we have that $R$ is a right coherent ring and $T_{R}$ is finitely presented. From Proposition \ref{Angeleri_Proposition_2.6_1}, the class $\CY_{R}$ consists of finitely presented modules belonging to the class $\CB_{\zeta}$. Hence $\CY_{R} \subseteq \CY_{R}^{\prime}$. Therefore $\CY_{R} = \CY_{R}^{\prime}$. Similarly, we can prove that ${_{S}\CY} = {_{S}\CY^{\prime}}$ and, by Proposition \ref{rezultat_0}, we have that $\Delta_{R} : \CY_{R}^{\prime} \rightleftarrows {_{S}\CY^{\prime}} : \Delta_{S}$ is a duality.
\end{proof}

\begin{proposition}\label{rezultat_3} Let $_{S}T_{R}$ be a finitely cosilting bimodule. Then:
\begin{itemize}
   \item[(1)] $\Gamma_{S}\Delta_{R}$ carries finitely generated modules to zero;
   \item[(2)] $\Gamma_{R}\Delta_{S}$ carries finitely generated modules to zero.
\end{itemize}
\end{proposition}
\begin{proof}(1) Suppose that $T_{R}$ (respectively, $_{S}T$) is a finitely cosilting module with respect to the injective copresentation $\zeta : Q_{0} \to Q_{1}$ (respectively, $\tau : S_{0} \to S_{1}$). Since $_{S}T_{R}$ is finitely cosilting bimodule, it follows, by Corollary \ref{Angeleri_Prposition_4.3_1}, that $R$ is a right coherent ring and $T_{R}$ is a finitely presented module. Analogous, it can be proved that $S$ is a left coherent ring and $_{S}T$ is finitely presented. By Proposition \ref{Angeleri_Proposition_2.6_1}, we have that the class $\CY_{R}$ (respectively, the class $_{S}\CY$) consists of finitely presented modules belonging to the class $\CB_{\zeta}$ (respectively, to the class $\CB_{\tau}$).

Let $X$ be a finitely generated $R$-module and let $X^{\prime} = \Rej_{T}(X)$. By Lemma \ref{Angeleri_Lemma_2.5}, $X/X^{\prime} \in \CY_{R}$ and $\Delta_{R}(X) \cong \Delta_{R}(X/X^{\prime})$. Since $\Delta_{R}(X/X^{\prime}) \in {_{S}\CY}$ (by Proposition \ref{rezultat_0}), hence $\Delta_{R}(X/X^{\prime}) \in \CB_{\tau}$, we have that $\Delta_{R}(X/X^{\prime}) \in {^{\perp}{_{S}T}}$, i.e. $\Gamma_{S}(\Delta_{R}(X/X^{\prime})) = 0$, and thus $\Gamma_{S}(\Delta_{R}(X)) = 0$.

(2) Similarly with (1).
\end{proof}

For a bimodule $_{S}T_{R}$, we set the following classes $\CX_{R}^{\prime} = \Ker\Delta_{R} \cap \presr R$ and $_{S}\CX^{\prime} = \Ker\Delta_{S} \cap S \presl$.

\begin{proposition}\label{rezultat_4} Let $_{S}T_{R}$ be a finitely cosilting bimodule. Then $$\Gamma_{R} : \CX_{R}^{\prime} \rightleftarrows {_{S}\CX^{\prime}} : \Gamma_{S}$$ is a duality.
\end{proposition}
\begin{proof} Let $X \in \CX_{R}^{\prime}$. Then $X \in \Ker\Delta_{R}$ and $X$ is finitely presented $R$-module. By Proposition \ref{Angeleri_Proposition_1.2.2}, $\Gamma_{R}(X)$ is finitely presented, i.e. $\Gamma_{R}(X) \in S\presl$. From Proposition \ref{rezultat_2}, $\CY_{R} \subseteq \Refl_{\delta}$, so that all modules in $\CY_{R}$ are $\delta$-reflexive. Applying \cite[Lemma 4.1 (1)]{Angeleri_2000_2}, we obtain that $\Gamma_{R}(X) \in \Ker\Delta_{S}$. Thus $\Gamma_{R}(X) \in {_{S}\CX^{\prime}}$, which shows us that $\Gamma_{R}(\CX_{R}^{\prime}) \subseteq {_{S}\CX^{\prime}}$. By Proposition \ref{rezultat_3}, $\Gamma_{S}\Delta_{R}$ carries finitely generated modules to zero, hence, by \cite[Lemma 4.1 (3)]{Angeleri_2000_2}, we have that $X \cong \Gamma_{S}\Gamma_{R}(X)$. Thus $X \in \Refl_{\gamma}$, where $\gamma : X \to \Gamma_{S}\Gamma_{R}(X)$ is the associated natural transformation (for more details, see the proof of \cite[Lemma 4.1]{Angeleri_2000_2}), which show us that $\CX_{R}^{\prime} \subseteq \Refl_{\gamma}$.

Analogous, we can prove that $\Gamma_{S}(_{S}\CX^{\prime}) \subseteq \CX_{R}^{\prime}$ and $_{S}\CX^{\prime} \subseteq \Refl_{\gamma}$.
\end{proof}

\begin{proposition} Let $_{S}T_{R}$ be a finitely cosilting bimodule. Then:
\begin{itemize}
   \item[(1)] $\Delta_{S}\Gamma_{R}$ carries finitely presented modules to zero;
   \item[(2)] $\Delta_{R}\Gamma_{S}$ carries finitely presented modules to zero.
\end{itemize}
\end{proposition}
\begin{proof} Assume that $T_{R}$ (respectively, $_{S}T$) is finitely cosilting with respect to the injective copresentation $\zeta: Q_{0} \to Q_{1}$ (respectively, $\tau : S_{0} \to S_{1}$).

(1) By Corollary \ref{Angeleri_Prposition_4.3_1}, $R$ is a right coherent ring and $T_{R}$ is finitely presented. By Proposition \ref{Angeleri_Proposition_2.6_1}, the class $\CY_{R}$ consists of finitely presented modules belonging to the class $\CB_{\zeta}$.

Let $X$ be a finitely presented $R$-module. If we denote $\Rej_{T}(X)$ by $X^{\prime}$, we have, by Lemma \ref{Angeleri_Lemma_2.5}, that $X^{\prime} \in \CX_{R}$ and $\Gamma_{R}(X) \cong \Gamma_{R}(X^{\prime})$. We also have that $X$ is coherent, hence $X^{\prime}$ is finitely presented. It follows that $X^{\prime} \in \CX_{R}^{\prime}$. By Proposition \ref{rezultat_4}, $\Gamma_{R}(X^{\prime}) \in {_{S}\CX^{\prime}}$, hence $\Gamma_{R}(X^{\prime}) \in \Ker\Delta_{S}$. Therefore $\Gamma_{R}(X) \in \Ker\Delta_{S}$, i.e. $\Delta_{S}\Gamma_{R}(X) = 0$.

(2) It is proved similarly.
\end{proof}

Summarizing all the results above, now we are in position to state the main result of the paper, i.e. the following Cosilting Theorem.

\begin{theorem} \label{cosilting_theorem}
Let $_{S}T_{R}$ be a finitely cosilting bimodule. Then the following assertions hold.
\begin{itemize}
   \item[(1)] The pairs $(\CX_{R},\CY_{R})$ and $({_{S}\CX},{_{S}\CY})$ are torsion pairs in $\modr R$ and $S \lmod$, respectively.
   \item[(2)] $\Delta_{R} : \CY_{R}^{\prime} \rightleftarrows {_{S}\CY^{\prime}} : \Delta_{S}$ and $\Gamma_{R} : \CX_{R}^{\prime} \rightleftarrows {_{S}\CX^{\prime}} : \Gamma_{S}$ are dualities.
   \item[(3)] \begin{itemize}
                  \item[(i)] $\Gamma_{S}\Delta_{R}$ and $\Gamma_{R}\Delta_{S}$ carries finitely generated modules to zero.
                  \item[(ii)] $\Delta_{S}\Gamma_{R}$ and $\Delta_{R}\Gamma_{S}$ carries finitely presented modules to zero.
              \end{itemize}
\end{itemize}
\end{theorem}

%\section*{Acknowledgment}

\end{document}